\documentclass[12pt,reqno]{article}

\usepackage[usenames]{color}
\usepackage{enumerate}
\usepackage{amsmath}
\usepackage{amsthm}
\usepackage{amsfonts}

\usepackage[colorlinks=true,
linkcolor=webgreen,
filecolor=webbrown,
citecolor=webgreen]{hyperref}

\definecolor{webgreen}{rgb}{0,.5,0}
\definecolor{webbrown}{rgb}{.6,0,0}

\newcommand{\seqnum}[1]{\href{http://oeis.org/#1}{\underline{#1}}}

\newcommand{\ffloor}[1]{\left\lfloor #1 \right\rfloor}

\begin{document}

\theoremstyle{plain}
\newtheorem{theorem}{Theorem}
\newtheorem{corollary}[theorem]{Corollary}
\newtheorem{lemma}[theorem]{Lemma}
\newtheorem{proposition}[theorem]{Proposition}

\theoremstyle{definition}
\newtheorem{definition}[theorem]{Definition}
\newtheorem{example}[theorem]{Example}
\newtheorem{conjecture}[theorem]{Conjecture}
\newtheorem{problem}[theorem]{Problem}
\newtheorem{algorithm}[theorem]{Algorithm}

\theoremstyle{remark}
\newtheorem{remark}[theorem]{Remark}
\newtheorem{notation}[theorem]{Notation}

\newcommand{\ZZ}{{\mathbb Z}}

\begin{center}
\vskip 1cm{\LARGE\bf  A sequence of quasipolynomials arising from random numerical semigroups\\
\vskip 1cm}
\large
Calvin Leng\\
Department of Mathematics\\
University of California Davis\\
Davis, CA 95616\\
USA\\
\href{mailto:calleng@ucdavis.edu}{\tt calleng@ucdavis.edu} \\
\ \\
Christopher O'Neill\\
Department of Mathematics and Statistics\\
San Diego State University\\
San Diego, CA 92182\\
USA\\
\href{mailto:cdoneill@sdsu.edu}{\tt cdoneill@sdsu.edu}\\
\end{center}

\vskip .2 in

\begin{abstract}
A numerical semigroup is a subset of the non-negative integers that is closed under addition.  For a randomly generated numerical semigroup, the expected number of minimum generators can be expressed in terms of a doubly-indexed sequence of integers, denoted $h_{n, i}$, that count generating sets with certain properties.  We prove a recurrence that implies the sequence $h_{n,i}$ is eventually quasipolynomial when the second parameter is fixed.
\end{abstract}

\section{Introduction}
\label{sec:intro}

A \emph{numerical semigroup} is a subset of $\ZZ_{\geq 0}$ containing 0 that is closed under addition.  The \emph{numerical semigroup generated by a set} $A = \{a_1, a_2, \dots, a_k\}$ is the smallest numerical semigroup containing $A$, namely
\begin{align*}
S = \langle A \rangle = \langle a_1, \ldots, a_k \rangle = \{a_1 x_1 + \dots + a_k x_k : x_i \in \ZZ_{\geq 0}\}.
\end{align*}
A generating set $A$ is \emph{minimal} if for all $x \in A$, we have $\langle A \rangle \neq \langle A \setminus \{x\} \rangle$.  Every numerical semigroup $S$ has a unique minimal generating set, and the \emph{embedding dimension} of~$S$ is the size of its minimal generating set (see~\cite{numerical} for a thorough introduction). 

The authors of \cite{rnscomplex} introduce a model of randomly generating a numerical semigroup that is similar to the Erd\H os-Renyi model for random graphs.  Their model takes two inputs $M \in \ZZ_{\ge 1}$ and $p \in [0, 1]$, and randomly selects a generating set $A$ that includes each integer $n = 1, 2, \ldots, M$ with independent probability $p$.  

As an example, if $M = 40$ and $p = 0.1$, then one possible set is $A = \{6,9,18,20,32\}$ (this~is not unreasonable, as on average one would expect 4 generators to be selected).  However, only 3 elements of $A$ are minimal generators, since $18 = 9 + 9$ and $32 = 20 + 6 + 6$.  As such, the resulting semigroup $S = \langle A \rangle = \langle 6, 9, 20 \rangle$ has embedding dimension 3.  

One of the main results in \cite{rnscomplex} is that the expected number of minimal generators of a numerical semigroup $S$ sampled with the above model can be expressed as
\begin{align*}
\mathbb E [e(S)] = \sum_{n=1}^M p(1 - p)^{\ffloor{n/2}} \left(h_{n,0} + h_{n,1} p + h_{n,2} p^2 + \cdots\right),
\end{align*}
where $h_{n, i}$ equals the number of sets $A \subset [1, n/2) \cap \ZZ$ with $|A| = i$ that minimally generate a numerical semigroup not containing $n$.  Of interest is the asymptotic behavior of $\mathbb E [e(S)]$ for fixed $p$ as $M \to \infty$.  Although this is currently out of reach, $\mathbb E [e(S)]$ can be approximated for fixed $M$ using the above formula, so long as $h_{n,i}$ is known for $n \le M$.  

\begin{figure}[t]
\begin{center}
\begin{tabular}{r@{\,\,\,\,\,\,}r@{\,\,\,}r@{\,\,\,}r@{\,\,\,}r@{\,\,\,}r@{\,\,\,}r@{\,\,\,}r@{\,\,\,}r@{\,\,\,}r@{\,\,\,}r@{\,\,\,}r@{\,\,\,}r@{\,\,\,}r@{\,\,\,}r@{\,\,\,}l}
$n$ = 68:&  1, & 29, & 249, & 888, & 1705, & 2014, & 1599, & 888, & 347, & 91, & 14, & 1\phantom{,} \\

$n$ = 69:&  1, & 31, & 301, & 1181, & 2414, & 2939, & 2365, & 1335, & 535, & 147, & 25, & 2\phantom{,} \\

$n$ = 70:&  1, & 28, & 248, & 1012, & 2218, & 2873, & 2431, & 1414, & 569, & 155, & 26, & 2\phantom{,} \\

$n$ = 71:&  1, & 34, & 359, & 1577, & 3615, & 4945, & 4481, & 2878, & 1348, & 453, & 105, & 15, & 1\phantom{,} \\

$n$ = 72:&  1, & 25, & 222, & 893, & 1923, & 2498, & 2138, & 1267, & 526, & 147, & 25, & 2\phantom{,} \\

$n$ = 73:&  1, & 35, & 383, & 1764, & 4252, & 6139, & 5883, & 4008, & 2004, & 725, & 181, & 28, & 2\phantom{,} \\

$n$ = 74:&  1, & 34, & 337, & 1456, & 3361, & 4694, & 4365, & 2853, & 1345, & 453, & 105, & 15, & 1\phantom{,} \\

$n$ = 75:&  1, & 32, & 346, & 1582, & 3810, & 5567, & 5428, & 3758, & 1888, & 684, & 172, & 27, & 2\phantom{,} \\

$n$ = 76:&  1, & 33, & 334, & 1448, & 3413, & 5005, & 4992, & 3559, & 1863, & 705, & 181, & 28, & 2\phantom{,}
\end{tabular}
\end{center}
\caption{Values of $h_{n,i}$ for $n = 68$ through $n = 76$.}
\label{f:sequence}
\end{figure}

The doubly-indexed sequence $h_{n,i}$ is available on OEIS as \seqnum{A319608}, computed for $n \le 90$.  Figure~\ref{f:sequence} contains the values of $h_{n,i}$ for $n = 68, \ldots, 76$, where each row is comprised of $h_{n,0}, h_{n,1}, \ldots, h_{n,d_n}$ from left to right.   The following facts about the sequence $h_{n,i}$ are known:
\begin{itemize}
\item $h_{n,i}$ is nonzero for $n \ge 1$ and $0 \le i \le d_n = \ffloor{n/2} - \ffloor{n/3}$;
\item $h_{n,0} = 1$;
\item $h_{n,1} = \ffloor{(n + 1) / 2} - \tau(n)$, where $\tau(n)$ denotes the number of divisors of $n$; and
\item The sum of the $n^\text{th}$ row equals the number of irreducible numerical semigroups with Frobenius number $n$ \cite{irrnumsgpscount,fundamentalgaps}, which appears in OEIS as \seqnum{A158206} \cite{oeis}. 
\end{itemize}

Currently, computing the values of $h_{n, i}$ for large $n$ is time-intensive; the fastest known algorithm computes the $n^\text{th}$ row by first computing the set of irreducible numerical semigroups with Frobenius number $n$ and utilizing the last bullet point above~\cite{fundamentalgaps}.  This computation takes 3 days for $n = 89$ on the authors' machines.  The more values of $h_{n, i}$ that are known, the more accurately $\mathbb E [e(S)]$ can be approximated.  Due to the limited known values of $h_{n,i}$, appoximations computed with the currently known values still differ drastically from those obtained from experimental data.  

In this paper, we examine the combinatorics of the sequence $h_{n,i}$.  Our main result is Corollary~\ref{c:mainthm}, which states that for fixed $k$ the sequence $h_{n, d_n - k}$ coincides with a polynomial in $n \gg 0$ whose coefficients are 6-periodic functions of $n$, follows from the following recurrence.  

\begin{theorem}\label{t:mainthm}
Fix $k \in \ZZ_{\ge 0}$, $b \in \{0, 1, 2\}$, and $m > 24k + 12 - 8b$ with $m \equiv b \bmod 3$.  The recurrence
$$h_{n, d_n - k} = \sum_{l = 0}^k h_{m, d_m - l} \binom{d_n - d_m}{k - l}$$
holds for all $n \geq m$ satisfying $n \equiv b \bmod 3$.
\end{theorem}

A \emph{quasipolynomial} is a function $q:\ZZ \to \ZZ$ such that
\begin{align*}
q(x) = c_0(x) + c_1(x) x + c_2(x) x^2 + \dots + c_d(x) x^d
\end{align*}
where each $c_i(x)$ is a periodic function.  The \emph{degree} of $q$, denoted $\deg q$, is the largest integer $d$ for which $c_d$ is not identically 0, and the \emph{period} of $q$ is the smallest integer $p$ such that $c_i(x + p) = c_i(x)$ for every $x$ and $i$.  

\begin{corollary}\label{c:mainthm}
For fixed $k$, the function $n \mapsto h_{n, d_n - k}$ coincides with a quasipolynomial 
$$c_k(n)n^k + \cdots + c_1(n)n + c_0(n)$$
with degree $k$, period 6, and leading coefficient
\begin{align*}
c_k(n) =
\begin{cases}
\frac{2}{k!6^k}, & \text{if $n \equiv 0, 1 \bmod 3$;}\\
\frac{1}{k!6^k}, & \text{if $n \equiv 2 \bmod 3$,}
\end{cases}
\end{align*}
for all $n > 24k + 12 - 8b$, where $b \in \{0, 1, 2\}$ with $n \equiv b \bmod 3$.  
\end{corollary}

In the development of the proof of Theorem \ref{t:mainthm}, we obtain an algorithm for computing the values $h_{n, i}$ appearing in Corollary~\ref{c:mainthm} (Algorithm~\ref{a:hnkalgorithm}).  Our algorithm has obtained $h_{n, i}$ values that were previously unobtained.  With the improved algorithm and Theorem~\ref{t:mainthm}, explicit quasipolynomials have been provided for $h_{n, d_n - k}$ for each $k \le 7$ (see Figure~\ref{f:quasipolynomials} for the quasipolynomials up to $k = 4$).  Computing the quasipolynomial coefficients of $h_{n, d_n - 7}$ requires computing the value of e.g., $h_{183, d_{183} - 7} = h_{183, 23} = 6423209$, a task that would have been impossible with existing methods.

\begin{figure}
\centering
\begin{align*}
h_{n,d_n} &=
\begin{cases}
2, & \text{if $n \equiv 0 \bmod 6$ and $n \geq 18$;} \\
2, & \text{if $n \equiv 1 \bmod 6$ and $n \geq 7  $;} \\
1, & \text{if $n \equiv 2 \bmod 6$ and $n \geq 2  $;} \\
2, & \text{if $n \equiv 3 \bmod 6$ and $n \geq 15 $;} \\
2, & \text{if $n \equiv 4 \bmod 6$ and $n \geq 10 $;} \\
1, & \text{if $n \equiv 5 \bmod 6$ and $n \geq 5  $.}
\end{cases}\\
h_{n,d_n-1} &=
\begin{cases}
\frac{1}{3}(n + 3),  & \text{if $n \equiv 0 \bmod 6$ and $n \geq 42$;} \\
\frac{1}{3}(n + 11), & \text{if $n \equiv 1 \bmod 6$ and $n \geq 31$;} \\
\frac{1}{6}(n + 16), & \text{if $n \equiv 2 \bmod 6$ and $n \geq 26$;} \\
\frac{1}{3}(n + 6),  & \text{if $n \equiv 3 \bmod 6$ and $n \geq 39$;} \\
\frac{1}{3}(n + 8),  & \text{if $n \equiv 4 \bmod 6$ and $n \geq 34$;} \\
\frac{1}{6}(n + 19), & \text{if $n \equiv 5 \bmod 6$ and $n \geq 23$.}
\end{cases}\\
h_{n, d_n-2} &=
\begin{cases}
\frac{1}{36}(n^2 + 108),       & \text{if $n \equiv 0 \bmod 6$ and $n \geq 66$;} \\
\frac{1}{36}(n^2 + 16n + 19),  & \text{if $n \equiv 1 \bmod 6$ and $n \geq 55$;} \\
\frac{1}{72}(n^2 + 26n + 160), & \text{if $n \equiv 2 \bmod 6$ and $n \geq 50$;} \\
\frac{1}{36}(n^2 + 6n + 117),  & \text{if $n \equiv 3 \bmod 6$ and $n \geq 63$;} \\
\frac{1}{36}(n^2 + 10n - 20),  & \text{if $n \equiv 4 \bmod 6$ and $n \geq 58$;} \\
\frac{1}{72}(n^2 + 32n + 247), & \text{if $n \equiv 5 \bmod 6$ and $n \geq 47$.}
\end{cases}\\
h_{n, d_n - 3} &=
\begin{cases}
\frac{1}{648}(n^3 - 9n^2 + 342n - 3240),   & \text{if $n \equiv 0 \bmod 6$ and $n \geq 90$;} \\
\frac{1}{648}(n^3 + 15n^2 - 69n + 5885),   & \text{if $n \equiv 1 \bmod 6$ and $n \geq 79$;} \\
\frac{1}{1296}(n^3 + 30n^2 + 264n - 1952), & \text{if $n \equiv 2 \bmod 6$ and $n \geq 74$;} \\
\frac{1}{648}(n^3 + 315n - 2268),          & \text{if $n \equiv 3 \bmod 6$ and $n \geq 87$;} \\
\frac{1}{648}(n^3 + 6n^2 - 132n + 6200),   & \text{if $n \equiv 4 \bmod 6$ and $n \geq 82$;} \\
\frac{1}{1296}(n^3 + 39n^2 + 471n - 863),  & \text{if $n \equiv 5 \bmod 6$ and $n \geq 71$.}
\end{cases}\\
h_{n, d_n - 4} &=
\begin{cases}
\frac{1}{15552}(n^4 - 24n^3 + 828n^2 - 17280n + 419904),   & \text{if $n \equiv 0 \bmod 6$ and $n \geq 114$;} \\
\frac{1}{15552}(n^4 + 8n^3 - 282n^2 + 24728n + 413225),    & \text{if $n \equiv 1 \bmod 6$ and $n \geq 103$;} \\
\frac{1}{31104}(n^4 + 28n^3 + 204 n^2  - 10256n + 454912), & \text{if $n \equiv 2 \bmod 6$ and $n \geq 98 $;} \\
\frac{1}{15552}(n^4 - 12n^3 + 666n^2 - 12852n + 374949),   & \text{if $n \equiv 3 \bmod 6$ and $n \geq 111$;} \\
\frac{1}{15552}(n^4 -4n^3 - 300n^2 + 26528n - 490112),     & \text{if $n \equiv 4 \bmod 6$ and $n \geq 106$;} \\
\frac{1}{31104}(n^4 + 40n^3 + 510 n^2  - 8168n + 426817),  & \text{if $n \equiv 5 \bmod 6$ and $n \geq 95$.}
\end{cases}
\end{align*}
\caption{Quasipolynomial expressions for $h_{n, d_n - k}$ with $k = 0, 1, \ldots, 4$.}
\label{f:quasipolynomials}
\end{figure}

\section{Setup}
\label{sec:setup}

Unless otherwise stated, throughout the rest of the paper assume $n \in \ZZ_{\ge 1}$ and $b_n \in \{0, 1, 2\}$ with $n \equiv b_n \bmod 3$.  Let
$$X_n = \left(\tfrac{n}{3} , \tfrac{n}{2} \right) \cap \ZZ.$$

\begin{definition}\label{d:works}
Fix a set $A \subset \ZZ_{\ge 1}$.  We say $A$ \emph{works} for $n \in \ZZ_{\ge 1}$ if
\begin{enumerate}[(i)]
\item $n \notin \langle A \rangle$,
\item $x < n/2$ for all $x \in A$, and
\item $A$ minimally generates a numerical semigroup.
\end{enumerate}
In particular, $h_{n,i}$ equals the number of sets $A$ with $|A| = i$ that work for $n$.  
\end{definition}

To motivate the next several definitions, recall that for $n \ge 13$, 
\begin{align}\label{e:122}
h_{n, d_n} =
\begin{cases}
2, & \text{if $n \equiv 0, 1 \bmod 3$;}\\
1, & \text{if $n \equiv 2 \bmod 3$.}
\end{cases}
\end{align}
The set $X_n$ works for $n$ and $|X_n| = d_n$.  Let $E_{0, n}$ and $E_{1,n}$ denote the remaining working sets for $n$ of size $d_n$ when $b_n = 0$ and $b_n = 1$, respectively.  The key observation is that for all $n$,
\begin{align*}
X_n - \ffloor{\tfrac{n}{3}} &= \{1, 2, \ldots, d_n\}, \\
E_{0,n} - \ffloor{\tfrac{n}{3}} &= \{-1, 1, 3, 4, \ldots, d_n\}, \text{ and} \\
E_{1,n} - \ffloor{\tfrac{n}{3}} &= \{0, 2, 3, \ldots, d_n\}.
\end{align*}
Since $X_n$ contains every integer in the interval $(n/3, n/2)$, if we wanted to construct counted sets from $X_n$, we could only adjoin elements from $\{1, 2, \dots, \ffloor{n/3}\}$ to $X_n$. Thus we thought of $\ffloor{n/3}$ as a sort of cutoff point. From this, it felt natural to express sets in terms of how offset the elements are from $\ffloor{n/3}$. This motivates the following.

\begin{definition}\label{d:offset}
The \emph{offset form} of a set $A = \{x_1, x_2, \dots, x_k\} \subset \ZZ_{\ge 1}$ is the set
\begin{align*}
A_{(n)} = A - \ffloor{n/3} = \{x_1 - \ffloor{n/3}, x_2 - \ffloor{n/3}, \dots, x_k - \ffloor{n/3}\}.
\end{align*}
\end{definition}

After expressing the sets we computed in offset form, we noticed that we could go one step further. We noticed that if we instead expressed sets in terms of how different they are from $X_n$ and then take the offset form of the result, the expressions would be equal. This motivates the following.

\begin{definition}\label{d:ripair}
A set $I \subseteq \ZZ$ is an \emph{inserting set} for $n \in \ZZ_{\ge 1}$ if
$$I_{(n)} \subseteq \{-\ffloor{n/3}, \dots, -1, 0\},$$
and a set $R \subseteq \ZZ$ is a \emph{removing set} for $n$ if 
$$R_{(n)} \subseteq \{1, 2, \dots, d_n\}.$$
An \emph{RI-pair} for $n$ is a pair $(R, I)$ of a removing set $R$ and an inserting set $I$.

There is a natural bijection between $RI$-pairs for $n$ and the powerset of $\{1, 2, \dots, d_n\}$ given by the map
\begin{align*}
\phi(R, I) = (X_n \setminus R) \cup I.
\end{align*}
The inverse map is given by
\begin{align*}
A \mapsto (X_n \setminus A, A \setminus X_n).
\end{align*}
Since $\varphi_n$ gives a bijection between the two objects, we say the set \emph{corresponding to} an $RI$-pair $(R,I)$ is the set $\varphi_n(R, I)$, and vice-versa. 
\end{definition}

Theorem~\ref{t:mainthm} follows from the fact that for fixed $k$ and large $n$, every $RI$-pair $(R,I)$ corresponding to a working set for $n$ of size $d_n - k$ satisfies $I_{(n)} \subseteq \{p_n(k), \ldots, -1, 0\}$, where 
$$p_n(k) = b_n - 2k - 1$$
only depends on $n$ modulo 3 (Theorem~\ref{t:insertionlowerbound}).  As a consequence, the restrictions on removal sets corresponding to a given insertion set are independent of the size of $n$ in this case.  

\begin{example}\label{e:ripair}
If $n = 11$ and $k = 1$, then $h_{n,d_n-k} = h_{11,1} = 4$ and $X_n = \{4,5\}$.  The sets $A$ with $|A| = 1$ that work for $11$ are 
\begin{center}
\begin{tabular}{l@{\qquad}l}
$A = \{2\} = (X_n \setminus \{4,5\}) \cup \{2\}$, &
$A = \{4\} = (X_n \setminus \{5\}) \cup \{\}$, \\
$A = \{3\} = (X_n \setminus \{4,5\}) \cup \{3\}$, and &
$A = \{5\} = (X_n \setminus \{4\}) \cup \{\}$.
\end{tabular}
\end{center}
Theorem~\ref{t:classification} classifies the possible $RI$-pairs that correspond to working sets for large $n$.  
\end{example}

\section{Strongly bounded sets}
\label{sec:stronglyboundedsets}

We begin by classifying the working sets for $n$ that are strongly $n$-bounded (Definition~\ref{d:stronglynbounded}).  As it turns out, for $k$ fixed and large $n$, every working set for $n$ with size $d_n - k$ is strongly $n$-bounded (Theorem~\ref{t:quasilowerbound}).  Note that any strongly $n$-bounded set automatically satisfies parts~(ii) and~(iii) of Definition~\ref{d:works}.  

\begin{definition}\label{d:stronglynbounded}
We say a set $A \subset \ZZ_{\ge 1}$ is \emph{strongly $n$-bounded} if $A \subset (n/4, n/2)$.  
\end{definition}

\begin{proposition}\label{p:stronglynbounded}
A strongly $n$-bounded set $A$ works for $n$ if and only if $b_n \notin 3A_{(n)}$.  
\end{proposition}

\begin{proof}
Any strongly $n$-bounded set automatically satisfies part~(ii) and~(iii) of Definition~\ref{d:works} since $x + y > \tfrac{n}{2} > z$ for any $x, y, z \in A$.  As such, $A$ works for $n$ if and only if $n \notin \langle A \rangle$.  Moreover, since $A$ is strongly $n$-bounded, we have $x < n < y$ for any $x \in 2A$ and $y \in 4A$, so $n \in \langle A \rangle$ if and only if $n \in 3A$.  The claim now follows from the fact that $n = 3 \ffloor{n/3} + b_n$.  
\end{proof}

\begin{definition}\label{d:removaldegree}
An $RI$-pair $(R,I)$ is \emph{compatible} for $n$ (or, equivalently, $R$ is \emph{compatible with}~$I$) if the corresponding set $A$ satisfies $b_n \notin 3A_{(n)}$.  The \emph{removal degree} of an inserting set $I$, denoted $r(I)$, is given by
\begin{align*}
r(I) = \min \{|R| : (R, I) \text{ is } \text{compatible}\}.
\end{align*}
and the \emph{removal degree} of an integer $\alpha \leq 0$ is given by $r(\alpha) = r(\{\alpha\})$.  
\end{definition}

\begin{remark}\label{r:removaldegree}
Note that Proposition~\ref{p:stronglynbounded} does \textbf{not} imply that an $RI$-pair $(R, I)$ compatible for~$n$ corresponds to a set $A$ that works for $n$, as $A$ need not be strongly $n$-bounded in general.  
\end{remark}

Theorem~\ref{t:classification} classifies the $RI$-pairs compatible for $n$ in terms of $I_{(n)}$ and $R_{(n)}$ by examining the different ways for three integers to sum to $b_n \in \{0, 1, 2\}$. 

\begin{theorem}\label{t:classification}
If $A$ is a set and $(R,I)$ is the corresponding $RI$-pair, then $b_n \notin 3A_{(n)}$ if and only if for all $\alpha \in I_{(n)}$, the following hold:
\begin{enumerate}[(i)]
\item\label{t:classification:form1}
$b_n - 2\alpha \in R_{(n)}$;

\item\label{t:classification:form2}
$(b_n - \alpha)/2 \in R_{(n)}$ if $\alpha \equiv b_n \bmod 2$;

\item\label{t:classification:form3}
$b_n - \alpha - \beta \in R_{(n)}$ for all $\beta \in I_{(n)}$ with $\beta \ne \alpha$; and

\item\label{t:classification:form4}
$y \in R_{(n)}$ or $b_n - \alpha - y \in R_{(n)}$ for all $y$ satisfying $1 \le y < b_n - \alpha - y$.

\end{enumerate}
\end{theorem}

\begin{proof}
If any of \eqref{t:classification:form1}-\eqref{t:classification:form4} is violated for some $\alpha \in I_{(n)}$, then it is easy to check that $b_n \in 3A_{(n)}$.  Conversely, suppose $b_n \in 3A_{(n)}$, meaning $\alpha + y + z = b_n$ for some $\alpha, y, z \in 3A_{(n)}$ with $\alpha \le y \le z$.   Since $b_n \in \{0, 1, 2\}$, we must have $\alpha \leq 0$ and thus $\alpha \in I_{(n)}$.  Let $S = \{\alpha, y, z\}$.  If every element of $S$ is nonpositive, then $\alpha = y = z = b_n = 0$ so \eqref{t:classification:form1} fails to hold and we are done.  As such, at most 2 elements of $S$ are nonpositive, so $z > 0$.  Similarly, if $|S| = 1$, then $\alpha = y = z = b_n = 0$ so \eqref{t:classification:form1} fails to hold and we are done.  This leaves four distinct cases:
\begin{itemize}
\item 
$|S| = 2$ and $y \leq 0$, in which case $y = \alpha$ and \eqref{t:classification:form1} fails to hold;

\item 
$|S| = 2$ and $y > 0$, in which case $y = z$ and \eqref{t:classification:form2} fails to hold;

\item 
$|S| = 3$ and $y \leq 0$, in which case $\alpha \neq y$ and \eqref{t:classification:form3} fails to hold; or

\item 
$|S| = 3$ and $y > 0$, in which case $y \neq z$ and \eqref{t:classification:form4} fails to hold.
\end{itemize}
This completes the proof.  
\end{proof}

\begin{remark}
Given an inserting set $I$, Theorem~\ref{t:classification} provides a systematic way to construct a removing set $R$ such that the set $A$ corresponding to $(R,I)$ satisfies $b_n \notin 3A_{(n)}$.  Most applications of Theorem~\ref{t:classification} will involve starting with a set $R = \emptyset$ and systematically putting elements into $R$; see Example~\ref{e:findingremovingset}.  Moreover, Theorem~\ref{t:classification} yields a better-than-brute-force method of computing $h_{n, d_n - k}$ for large $n$; see Algorithm~\ref{a:hnkalgorithm}.  
\end{remark}

\begin{lemma}\label{l:removalformula}
We have
\begin{align*}
r(\alpha) = 1 + \left\lceil \frac{b_n - \alpha - 1}{2} \right\rceil
\end{align*}
for any integer $\alpha \leq 0$.
\end{lemma}

\begin{proof}
Fix $\alpha \leq 0$, let $I = \{\alpha\}$, and suppose $R$ is a removing set that is minimal among all removing sets compatible with $I$.  We will apply Theorem~\ref{t:classification}, noting that for fixed $\alpha$, parts~\eqref{t:classification:form1}-\eqref{t:classification:form4} each require distinct elements to lie in $R_{(n)}$.  Theorem~\ref{t:classification}\eqref{t:classification:form1} requires 1 element to lie in $R$, and Theorem~\ref{t:classification}\eqref{t:classification:form4}, which forces 
$\lfloor (b_n - \alpha - 1)/2 \rfloor$
additional elements to lie in $R$.  Since $|I| = 1$, Theorem~\ref{t:classification}\eqref{t:classification:form3} is vacuously satisfied.  This leaves Theorem~\ref{t:classification}\eqref{t:classification:form2}, which only requires an additional element to lie in $R$ if $\alpha \equiv b_n \bmod 2$.  This completes the proof.  
\end{proof}

\begin{lemma}\label{l:nofreeinserting}
If $A \subset \ZZ_{\ge 1}$ corresponds to an $RI$-pair $(R,I)$ that is compatible for $n$, then
\begin{align*}
|A| \leq d_n + 1 - r(m),
\end{align*}
where $m = \min I_{(n)}$.  
\end{lemma}

\begin{proof}
Let $m = \min I_{(n)}$, and apply Theorem~\ref{t:classification} to $\alpha = m$.  Following the proof of Lemma~\ref{l:removalformula}, Theorem~\ref{t:classification}\eqref{t:classification:form1},~\eqref{t:classification:form2}, and~\eqref{t:classification:form4} require $r(m)$ elements to lie in $R$, and Theorem~\ref{t:classification}\eqref{t:classification:form3} requires $R$ to contain an additional $|I| - 1$ elements.  We conclude
\begin{align*}
|A| = d_n + |I| - |R| \leq d_n + |I| - r(m) - |I| + 1 = d_n + 1 - r(m),
\end{align*}
as desired.
\end{proof}

\begin{theorem}\label{t:insertionlowerbound}
Fix $k \in \ZZ_{\ge 0}$, and suppose $A \subset \ZZ_{\ge 1}$ corresponds to an $RI$-pair $(R,I)$ that is compatible for $n$.  If $|A| \geq  d_n - k$, then
\begin{align*}
&I_{(n)} \subset \{p_n(k), p_n(k) + 1, \dots, -1, 0\}.
\end{align*}
\end{theorem}

\begin{proof}
Let $m = \min I_{(n)}$.  By Lemma~\ref{l:nofreeinserting}, we have
\begin{align*}
d_n - k \leq |A| \leq d_n + 1 - r(m),
\end{align*}
meaning $k \geq r(m) - 1$.  Applying Lemma~\ref{l:removalformula}, we obtain
\begin{align*}
k \ge 1 + \left\lceil \frac{b_n - m - 1}{2} \right\rceil - 1 \ge \frac{b_n - m - 1}{2} 
\end{align*}
which can then be rearranged to yield $m \geq b_n - 2k - 1 = p_n(k)$.
\end{proof}

\begin{theorem}\label{t:quasilowerbound}
If $n > 24k + 12 - 8b_n$, then every set $A$ with $|A| = d_n - k$ that works for $n$ is strongly $n$-bounded.
\end{theorem}

\begin{proof}
Fix a set $A$ with $|A| = d_n - k$ that works for $n$.  Theorem~\ref{t:insertionlowerbound} implies 
\begin{align*}
\min A - p_n(k)
\geq \ffloor{\frac{n}{3}}
= \frac{n - b_n}{3}
= \frac{\tfrac{1}{4}n - b_n}{3} + \frac{n}{4}
> \frac{6k + 3 - 3b_n}{3} + \frac{n}{4}
= 2k + 1 - b_n + \frac{n}{4}
\end{align*}
meaning $A$ is strongly $n$-bounded.  
\end{proof}

\section{Proof of Theorem~\ref{t:mainthm}}
\label{sec:mainresults}

We now have enough machinery to prove Theorem~\ref{t:mainthm} and Corollary~\ref{c:mainthm}.

\begin{proof}[Proof of Theorem~\ref{t:mainthm}]
Fix $n, m \in \ZZ$ satisfying $n \equiv m \bmod 3$ and $n \ge m > 24k + 12 - 8b_n$.  Let $\mathcal S_k$ denote the set of $k$-subsets of $\{d_m + 1, d_m + 2, \dots, d_n\}$.  We will prove the claim combinatorially by constructing a bijection between
\begin{align*}
\mathcal A := \{A : \text{$A$ works for $n$ and $|A| = d_n - k$}\}
\end{align*}
and
\begin{align*}
\mathcal B := \bigcup_{l = 0}^k \big( \{A : \text{$A$ works for $m$ and $|A| = d_m - l$}\} \times \mathcal S_{k-l} \big).
\end{align*}
Fix $A \in \mathcal A$, and let $(R,I)$ denote the corresponding $RI$-pair.  Write $R = R_1 \cup R_2$ with
\begin{align*}
&(R_1)_{(n)} = \{\alpha \in R_{(n)} : 1 \leq \alpha \leq d_m\} \text{ and} \\
&(R_2)_{(n)} = \{\alpha \in R_{(n)} : d_m + 1 \leq \alpha \leq d_n\},
\end{align*}
and define $f:\mathcal A \to \mathcal B$ by
\begin{align*}
f(A) = (\varphi_m(R_1, I), (R_2)_{(n)}).
\end{align*}

We first show $f$ is well-defined.  Let $l = |R_1| - |I|$.  It is clear that $(R_2)_{(n)} \in \mathcal S_{k-l}$, and $(R_1,I)$ is an $RI$-pair for~$m$, so it remains to show that $(R_1, I)$ is compatible for $m$.  By~Theorem~\ref{t:quasilowerbound}, $A$~is strongly $n$-bounded, so Theorem~\ref{t:insertionlowerbound} implies $\min I_{(n)} \ge p_n(k) = p_m(k)$.  The key observation is that the criteria in Theorem~\ref{t:classification}\eqref{t:classification:form1}-\eqref{t:classification:form4} only involve $I_{(n)}$ and $R_{(n)}$, so tracing through each part, the fact that $(R,I)$ is compatible for $n$ implies $(R_1, I)$ is compatible for $m$.  Hence, $f$ is well-defined.  

To prove $f$ is a bijection, we observe that basic set-theoretic arguments verify the map
\begin{align*}
((R_1, I), R_2) \mapsto (R_1 \cup R_2, I),
\end{align*}
is the inverse function of $f$, thereby completing the proof.  
\end{proof}

\begin{proof}[Proof of Corollary \ref{c:mainthm}]
Fix $k \ge 0$ and $b \in \{0, 1, 2\}$, and let
$$m = \min \{x > 24k + 12 - 8b : x \equiv b \bmod 3\}.$$
For any $n \ge m$ satisfying $n \equiv b \bmod 3$, we obtain the expression 
\begin{align}\label{eq:mainthmquasi}
h_{n, d_n - k} = h_{m, d_m - k} \binom{d_n - d_m}{0} + h_{m, d_m - k + 1} \binom{d_n - d_m}{1} + \cdots + h_{m, d_m} \binom{d_n - d_m}{k}
\end{align}
from Theorem~\ref{t:mainthm}, wherein each binomial coefficient is a polynomial in $d_n$ of degree at most $k$.  Since $d_n$ is a quasilinear function of $n$ with period 6, we conclude $h_{n, d_n - k}$ is a quasipolynomial in $n$ of degree $k$ and period 6.  

It remains to verify the leading coefficient of $h_{n, d_n-k}$ has the desired form.  The highest degree term in \eqref{eq:mainthmquasi} is
\begin{align*}
h_{m, d_m} \binom{d_n - d_m}{k} &= h_{m, d_m} \frac{(d_n - d_m) \cdot (d_n - d_m - 1) \cdots (d_n - d_m - k)}{k!}.
\end{align*}
Combined with the fact that $d_n$ has constant leading coefficient $1/6$, we obtain the leading coefficient $h_{m, d_m}/k! 6^k$, and the claim now follows from examination of Figure~\ref{f:quasipolynomials}.  
\end{proof}

\begin{example}\label{e:mainthm}
For fixed $k \ge 1$, the proof of Corollary~\ref{c:mainthm} provides a slightly optimized method of computing the eventual quasipolynomial form of $h_{n, d_n-k}$.  For example, let $k = 3$ and $b = 0$.  In this case, $m = 87$, and consulting \seqnum{A319608} we see
\begin{align*}
h_{87, d_{87} - 0} = 2, && h_{87, d_{87} - 1} = 31, && h_{87, d_{87} - 2} = 228, && \text{and} && h_{87, d_{87} - 3} = 1055.
\end{align*}
From here, we obtain
\begin{align*}
h_{n, d_n - 2} &= 2 \cdot \binom{d_n - d_{87}}{3} + 31 \cdot \binom{d_n - d_{87}}{2} \\ &\ \ \ + 228 \cdot \binom{d_n - d_{87}}{1} + 1055 \cdot \binom{d_n - d_{87}}{0}
\end{align*}
for all $n \equiv 0 \bmod 3$ such that $n \geq 87$.  Expanding binomial coefficients yields
\begin{align*}
h_{n, d_n - 3} = \frac{1}{6}(2d_n^3 + 3d_n^2 + 19d_n - 12),
\end{align*}
and substituting
\begin{align*}
d_n = \ffloor{\frac{n-1}{2}} - \ffloor{\frac{n}{3}} = 
\begin{cases}
\frac{1}{6}n - 1, & \text{if $n \equiv 0 \bmod 6$;} \\
\frac{1}{6}n - \frac{1}{2}, & \text{if $n \equiv 3 \bmod 6$.}
\end{cases}
\end{align*}
into the expression for $h_{n, d_n - 3}$, we arrive at
\begin{align*}
h_{n, d_n - 3} =
\begin{cases}
\frac{1}{648}(n^3 - 9n^2 + 342n - 3240), & \text{if $n \equiv 0 \bmod 6, n \ge 87$;} \\
\frac{1}{648}(n^3 + 315n - 2268), & \text{if $n \equiv 3 \bmod 6, n \ge 90$.}
\end{cases}
\end{align*}
Repeating this process for $b = 1, 2$ yields the function given in Figure~\ref{f:quasipolynomials}.  
\end{example}

\begin{remark}
It is interesting to note that the eventual quasipolynomial form of $h_{n, d_n - 3}$ would not be impossible to compute using the ``standard'' method of finding polynomial coefficients.  Indeed, the values of $h_{n,i}$ have only been successfully computed for $n \le 90$, and since the quasipolynomial behavior of $h_{n, d_n - 3}$ only holds for $n \geq 87$, the standard methods of finding the coefficients of a cubic require knowning $h_{87, d_{87} - 3}, h_{90, d_{90} - 3}, \ldots$, most of which have yet to be computed.  The above method, on the other hand, only relies on $h_{87,d_{87}-i}$ for $0 \le i \le 3$.  
\end{remark}

\begin{algorithm}\label{a:hnkalgorithm}
Theory developed in Section~\ref{sec:stronglyboundedsets} yields an algorithm to compute $h_{m, d_m - k}$ for $m > 24k + 12 - 8b_m$.  In particular, for each possible inserting set $I \subset \{p_m(k), \ldots, -1, 0\}$ for~$m$, Theorem~\ref{t:classification} determines precisely which removal sets $R$ are compatible with $I$.  Example~\ref{e:findingremovingset} demonstrates the main idea of the algorithm.  

The authors used a C++ implementation, now posted on Github at the following URL, to compute the quasipolynomial functions in Corollary~\ref{c:mainthm} up to $k = 7$, the last of which took 6 hours to complete.  
\begin{center}
\url{https://github.com/calvinleng97/rnsg-qp-coeffs}
\end{center}
\end{algorithm}

\begin{example}\label{e:findingremovingset}
Suppose $n = 60$, and consider the insertion set $I = \{17, 18\}$.  Theorem~\ref{t:classification} provides a systematic method of constructing all removal sets $R$ that are compatible with $I$.  Since $\min I > n / 4$, the resulting set will be strongly $n$-bounded. This ensures the resulting sets $A$ corresponding to $(R,I)$ will work for $n$.  

We check every item of Theorem \ref{t:classification} with every element $\alpha \in I_{(n)}$ to construct $R_{(n)}$.  We first compute the offset form
\begin{align*}
I_{(n)} = \{-3, -2\}
\end{align*}
and initialize $R_{(n)} = \emptyset$. Note that $b_n = 0$ since $60 \equiv 0 \bmod 3$.  We begin by applying Theorem~\ref{t:classification}(\ref{t:classification:form1})-\eqref{t:classification:form3} to each $\alpha \in I_{(n)}$, since Theorem~\ref{t:classification}(\ref{t:classification:form4}) requires additional decisions.  For $\alpha = -2$, we see that $1, 4, 5 \in R_{(n)}$, and for $\alpha = -3$, we must have $R_{(n)} = \{1, 4, 5, 6\}$.  Lastly, we deal with Theorem \ref{t:classification}(\ref{t:classification:form4}), which is vacuously satisfied for $\alpha = -2$, and for $\alpha = -3$ implies either $1 \in R_{(n)}$ or $2 \in R_{(n)}$, the first of which is already required from above.  As such, $R_{(n)} = \{1,4,5,6\}$ yields a removal $R = \{1,4,5,6\}$ that is compatible with $I$.  Moreover, any removal set $R' \supset R$ is also compatible with $I$.  
\end{example}

\section{Future Work}
\label{sec:futurework}

Although the quasipolynomials in Corollary~\ref{c:mainthm} only hold for $n$ sufficiently large, the machinery developed in Section~\ref{sec:stronglyboundedsets} describes the sets counted by $h_{n,d_n-k}$ and the relations between them as $n$ varies.  For~$n$ just below the start of quasipolynomial behavior, computations indicate the sets counted by $h_{n,d_n-k}$ are simply those predicted by Theorem~\ref{t:classification} that still mimimally generate a numerical semigroup.  A better understanding of this phenomenon could allow Algorithm~\ref{a:hnkalgorithm} to be extended to all $n \ge 1$, rather than just sufficiently large $n$.  

\begin{problem}\label{prob:backtracking}
Characterize the sets counted by $h_{n,d_n-k}$ for all $n$ in terms of those counted by $h_{n,d_n-k}$ for $n$ sufficiently large.  
\end{problem}

Algorithm~\ref{a:hnkalgorithm} has the potential to be parallelized (with different threads handling different insertion sets), but the current implementation does not take advantage of this fact.  Doing so would likely extend the current limits of computation, which would be especially useful if Problem~\ref{prob:backtracking} has a positive answer.  

\begin{problem}\label{prob:parallel}
Write a parallelized implementation of Algorithm \ref{a:hnkalgorithm}.
\end{problem}

\bigskip
\hrule
\bigskip

\noindent 2010 {\it Mathematics Subject Classification}:\
Primary 20M14; Secondary 05E40.

\noindent \emph{Keywords}:\
numerical semigroup, quasipolynomial.

\bigskip
\hrule
\bigskip

\noindent (Concerned with sequences 
\seqnum{A158206}, 
\seqnum{A319608})

\bigskip
\hrule
\bigskip

\end{document}